\documentclass[11pt]{article}

\usepackage{amsthm,amsfonts,amssymb,epsfig,graphics,amsbsy,subfigure}
\usepackage{color}		
\usepackage{epsfig}

\usepackage{color}		
\usepackage{epsfig}
\usepackage{amssymb}
\usepackage{amsmath}
\usepackage{graphicx,subfigure}

\newcommand{\intR}{\int\limits_{\mathbb{R}} }

\newcommand{\intL}{\int\limits }
\newcommand{\half}{^\infty_0 }
\newcommand{\RR}{\mathbb{R}} 
\newcommand{\xx}{\mathbf{x}} 
\newcommand{\yy}{\mathbf{y}} 
\newcommand{\zz}{\mathbf{z}} 
\newcommand{\uu}{\mathbf{u}} 
\newcommand{\aalpha}{\boldsymbol{\alpha}} 

\newtheorem{thm}{Theorem}[section]
\newtheorem{defi}[thm]{Definition}

\newtheorem{cor}[thm]{Corollary}

\newtheorem{prop}[thm]{Proposition}

\title{Inversion of the elliptical Radon transform arising in migration imaging using the regular Radon transform}

\author{Sunghwan Moon$^a$\footnote{shmoon@unist.ac.kr} and Joonghyeok Heo$^b$}
\date{\small $^a$Department of Mathematical Sciences, Ulsan National Institute of Science and Technology (UNIST), Ulsan 44919, Korea \\
$^b$UM Energy Institute, 
University of Michigan, Ann Arbor, MI 48109, U.S.}

\begin{document}
\maketitle
\begin{abstract}
 In recent years, many types of elliptical Radon transforms that integrate functions over various sets of  ellipses/ellipsoids have been considered, relating to studies in bistatic synthetic aperture radar, ultrasound reflection tomography, radio tomography, and migration imaging. 
In this article, we consider the transform that integrates a given function in $\RR^n$ over a set of ellipses (when $n=2$) or ellipsoids of rotation (when $n\geq 3$) with foci restricted to a hyperplane. 
We show a relation between this elliptical Radon transform and the regular Radon transform, and provide the inversion formula for the elliptical Radon transform using this relation.
Numerical simulations are performed to demonstrate the suggested algorithms in two dimensions, and these simulations are also provided in this article. 

\end{abstract}
\textbf{Keywords:} 
Radon transform; tomography; elliptical; migration; seismic reflection

\textbf{AMS Subject Classification:} 44A12; 65R10; 92C55



\section{Introduction}

Radon-type transforms that assign to a given function its integrals over various sets of ellipses/ellipsoids arise in migration imaging under an assumption that the medium is acoustic and homogeneous. 
The aim of migration is to construct an image of the inside of the earth from seismic reflections recorded at its surface~\cite{gazdags84,robinson83}.
A graphical approach called classical migration was developed systematically by Hagedoorn \cite{hagedoorn54}.
Classical migration had been abandoned after the wave-equation method was introduced by Claerbout~\cite{claerbout71} in 1971. 
Gazdag and Sguazzero pointed out that the classical migration procedures that existed at that time were not based on a completely sound theory~\cite{gazdags84}.
However, a correct construction for the wave-equation method was often difficult to find because the experiment data did not fit into a single wave equation.
To adapt the diffraction stack to borehole seismic experiments, a new approach to seismic migration was found.
This approach gave classical migration a sound theory.
After that discovery, classical migration has attracted many researchers in the field.
The underlying idea is that seismic data in the far field can be regarded as if the data are coming from integrals of the earth's acoustic scattering potential over surfaces determined by the velocity model~\cite{millerob87}.
These Radon-type transforms relate to migration imaging as well as Bistatic Synthetic Aperture Radar (BiSAR) \cite{ambartsoumianfknq11,cokert07,krishnanq11,krishnanlq12,yarmanyc08}, Ultrasound Reflection Tomography (URT) \cite{ambartsoumiankq11,gouiaa12,roykcm14}, and radio tomography \cite{wilsonp09,wilsonp10,wilsonpv09}.

Because of these applications, there have been several papers devoted to the topic of elliptical Radon transforms. 
The family of ellipses with one focus fixed at the origin and the other one moving along a given line was 
considered in~\cite{krishnanlq12}. In the same paper, the family of ellipses with a fixed focal distance was 
also studied. 
The authors of~\cite{ambartsoumiankq11,gouiaa12} dealt with the case of circular acquisition, when the two foci of ellipses 
with a given focal distance are located on a given circle. A family of ellipses with two moving foci was 
also handled in~\cite{cokert07}. 
Radio tomography, which uses a wireless network of radio transmitters and receivers to image the distribution of 
attenuation within the network, was discussed in~\cite{wilsonp09,wilsonp10,wilsonpv09}.
They approximated the obtained signal by the volume integral of the attenuation over this ellipsoid.
One work \cite{smoon,moonert14} derived two inversion formulas of this volume integral of the attenuation over this ellipsoid and studied its properties.
Many works found an approximate inversion for elliptical Radon transforms. 

Here we consider migration imaging and introduce a new type of an elliptical Radon transform obtained by restricting the position of the source and receiver in migration imaging.
We find an explicit inversion formula for this elliptical Radon transform arising in migration imaging which is the line/area integral of the function over the ellipse/ellipsoid with foci restricted to a hyperplane.

The rest of this paper is organized as follows. The problem of interest is stated precisely and the elliptical Radon transform is formulated in section~\ref{formulation}. In section 3, we show how to reduce the elliptical Radon transform to the regular Radon transform.
The numerical simulation to demonstrate the suggested 2-dimensional algorithm is presented in section 4.

\section{Formulation of the problem}\label{formulation}
Let $\mathbf s\in\RR^3$ and $\mathbf r\in\RR^3$ represent 3-dimensional source and receiver positions, respectively.
For fixed points $(\mathbf s,\mathbf r)$, an isochron surface $I_{(\mathbf s,\mathbf r,t)}$ \footnote{Hagedoorn called this set a surface of maximum concavity (see~\cite{hagedoorn54}).} is a surface consisting of image points $\xx=(x_1,x_2,x_3)\in\RR^3$ associated by the travel time function $\tau(\xx,\yy)$, which gives the travel time $t\in[0,\infty)$ between two points $\xx\in\RR^3$ and $\yy\in\RR^3$ with a known velocity $v(\xx)$ (see \cite{millerob87}).

Mathematically, $I_{(\mathbf s,\mathbf r,t)}$ can be described as the set of the image points $\xx$ satisfying the constraint that the total travel time from the source $\mathbf s$ though the image point $\xx$ to the receiver $\mathbf r$ is constant and equal to $t$.
The isochron surface $I_{(\mathbf s,\mathbf r,t)}$ can be represented as
$$
I_{(\mathbf s,\mathbf r,t)}=\{\xx\in\RR^3:t=\tau(\xx,\mathbf s)+\tau(\xx,\mathbf r)\}.
$$

Seismic experiments yield data $g(\mathbf s,\mathbf r,t)$ which are functions of the source position $\mathbf s$, receiver position $\mathbf r$, and time $t$.
Assuming an object function $f(\xx)$ on $\RR^3$, the data $g$ is modelled by the integral of $f$ over $I_{(\mathbf s,\mathbf r,t)}$, i.e.,
$$
g(\mathbf s,\mathbf r,t)=\intL_{I_{(\mathbf s,\mathbf r,t)}}f(\xx)d\xx\quad(\mbox{see \cite{millerob87}}).
$$
The two detectors have nonzero sizes and time is also passing, so it is reasonable to assume that what we measure is the ``average''-concentrated near the location of the detectors and nearly zero sufficiently far away from the detectors-over a small region of space and a small time interval preceding the time $t$.
In mathematical terms, our data can be written as
$$
g(\mathbf s,\mathbf r,t)=\intL_{\RR^3}f(\xx)\delta(\tau(\xx,\mathbf s)+\tau(\xx,\mathbf r)-t)d\xx,
$$
where $\delta$ is the Dirac delta function. 
(Most works~\cite{ambartsoumiankq11,ambartsoumianfknq11,cokert07,gouiaa12,krishnanq11,krishnanlq12} dealing with an elliptical Radon transform consider the arc length measure instead.)
Miller, Oristaglio, and Beylkin suggested this model and approximately recovered the object function $f$ by appropriately weighted back projection of the data in \cite{millerob87}.
However, we present the explicit formula for reconstructing $f$ using the specific minimum data of $g$ by changing the variables of positions of the two detectors.

The total dimension of the data $g$ is 7. To reduce the overdeterminacy, we assume that $\mathbf s$ and $\mathbf r$ are located on a hyperplane, say $x_3=0$ and also that the difference vector $\mathbf {s-r}$ is parallel to a given line, say the $x_1$-axis. (In 2 dimensions, the difference vector $\mathbf {s-r}$ is automatically parallel to the $x_1$-axis because the hyperplane is the line.)
Hence, for any $u_2\in\RR$, we write $\mathbf s=(s,u_2,0)$ and $\mathbf r=(r,u_2,0)$.
Also, assuming a velocity $v(\xx)$ is set to a constant 1, $\tau(\xx,\yy)$ becomes the distance $|\xx-\yy|$ between $\xx$ and $\yy$.
Since $\xx\in I_{(\mathbf s,\mathbf r,t)}$ satisfies $t=|\xx-(s,u_2,0)|+|\xx-(r,u_2,0)|$, a point $\xx\in I_{(\mathbf s,\mathbf r,t)}$ can be described by the ellipsoid equation
$$
\frac{(x_1-(s+r)/2)^2}{t^2/4}+\frac{(x_2-u_2)^2}{t^2/4-(r-s)^2/4}+\frac{x_3^2}{t^2/4-(r-s)^2/4}=1.
$$ 
Let us choose a constant $a$ between 0 and 1.
{\color{black}We consider the fixed value of $r$ satisfying the equation $r=s+ta$, which is a new additional requirement.}
Then a point $\xx\in I_{(\mathbf s,\mathbf r,t)}$ can be also described by the equation
$$
\frac{(x_1-{\color{black}(s+r)/2})^2}{1/4}+\frac{(x_2-u_2)^2}{(1-a^2)/4}+\frac{x_3^2}{(1-a^2)/4}=t^2.
$$ 
If we set $u_1=(s+r)/2$, then our data become
$$
\begin{array}{ll}
&g((u_1-ta,u_2,0),(u_1+ta,u_2,0),t)\\
&\displaystyle=\intL_{\RR^3}f(\xx) \delta\left(\sqrt{\frac{(x_1-u_1)^2}{1/4}+\frac{(x_2-u_2)^2}{(1-a^2)/4}+\frac{x_3^2}{(1-a^2)/4}}-t\right)d\xx\\
&:={\color{black}R_{E,a}}f(\uu,t),
\end{array}
$$
where $\uu=(u_1,u_2)\in\RR^2$.
{\color{black}Here, we restrict the positions of the source and receiver in migration imaging, say, $(u_1-ta,u_2,0)$ and $(u_1+ta,u_2,0)$ for $\uu\in\RR^2$.
In general, reducing the space of positions for two devices is more useful and economical.}
If the function is odd with respect to $x_3$, then ${\color{black}R_{E,a}}f$ is equal to zero. 
We thus assume the function $f$ to be even with respect to $x_3$: $f(\xx',x_3)=f(\xx',-x_3)$ where $\xx=(\xx',x_3)\in\RR^{2}\times\RR$.
We call ${\color{black}R_{E,a}}f$ an elliptical Radon transform, since this is the surface area integral of $f$ over the set of these ellipsoids.
We generalize this $3$-dimension setup to $n$-dimension and define a more general form of an elliptical Radon transform.
\begin{defi}
Let $a_1,a_2,\cdots,a_n>0$ be given numbers, let $A:=diag(a_1,\cdots,a_n)$ denote the $n\times n$ diagonal matrix with diagonal entries $a_i$, and let $f$ be a locally integrable function on $\RR^n$, even with respect to $x_n$.
The elliptical Radon transforms of $f$ is defined by
$$
{\color{black}R_{E,A}}f(\uu,t)=\intL_{\RR^n}f(\xx) \delta\left(\left|A^{-1}(\xx-(\uu,0))\right|-t\right)d\xx, \quad\mbox{for } (\uu,t)\in\RR^{n-1}\times[0,\infty).
$$
\end{defi}
Note that we do not need the condition $a_i<1$ any more.
If all $a_i$, $i=1,2,\cdots,n$, are equal to 1, the elliptical Radon transform ${\color{black}R_{E,A}}f$ becomes the spherical Radon transform with the centers of the sphere of integration located on the hyperplane, a well-studied transform (see~\cite{andersson88,beltukov09,couranth62,fawcett85,lavrentievrv70,klein03,narayananr10,nattererw01,nilsson97,norton80,palamodov04,reddingn01,yarmany07,yarmany11}).

\section{Inversion formula}\label{defpre}
Here we assume that the object function $f$ does not touch the detectors; that is, the support of $f$ does not intersect the hyperplane $x_n=0$ where the two detectors are located.

To obtain an inversion formula for the elliptical Radon transform, we manipulate ${\color{black}R_{E,A}}f$.
By changing the variables $A^{-1}(\xx-(\uu,0))\to\bar{\xx}$, we can write 
\begin{equation}\label{eq:defiref}
\begin{array}{ll}
{\color{black}R_{E,A}}f(\uu,t)&\displaystyle=|\mathbf a|_1\intL_{\RR^n}f(A\bar\xx +(\uu,0))\delta(|\bar\xx|-t)d\bar\xx\\
&\displaystyle=|\mathbf a|_1t^{n-1}\intL_{|\yy|=1}f(A \yy t+(\uu,0))dS(\yy)\\
&=\displaystyle2|\mathbf a|_1t^{n-1}\intL_{|\yy'|\leq 1}f(A (\yy',\sqrt{1-|\yy'|^2})t+(\uu,0))\frac{d\yy'}{\sqrt{1-|\yy'|^2}},
\end{array}
\end{equation}
where $\yy=(\yy',y_n)\in\RR^n$ and $|\mathbf a|_1=a_1a_2\cdots a_n$. 
Here $dS(\yy)$ is the area measure of the unit sphere.

Let $\mathfrak A=\{\zz=(\zz',z_n)=(z_1,z_2,\cdots,z_{n-1},z_n)\in\RR^n:0\leq |\zz'|^2\leq z_n\}$ and $\mathfrak B=\{\xx=(\xx',x_n)=(x_1,x_2,\cdots,x_{n-1},x_n)\in\RR^n:x_n\geq 0\}$.
We define a map ${\color{black}m_{n,A}}:\mathfrak A\to\mathfrak  B$ by
$$
{\color{black}m_{n,A}}(\zz)=(\bar A\zz',a_n\sqrt{z_n-|\zz'|^2}),
$$
where $\bar A:=diag(a_1,\cdots a_{n-1})$ is the $n-1\times n-1$ diagonal matrix.
\begin{prop}\label{prop:property}\indent
\begin{itemize}
\item The map ${\color{black}m_{n,A}}:\mathfrak A\to \mathfrak B$ is a bijection with the inverse map ${\color{black}m_{n,A}}^{-1}:\mathfrak B\to \mathfrak A$ defined by
$$
{\color{black}m_{n,A}}^{-1}(\xx)=(\bar A^{-1}\xx',|A^{-1}\xx|^2).
$$
\item We have 
$$
\begin{array}{ll}
\{{\color{black}m_{n,A}}^{-1}(\xx)\in \RR^n:|A^{-1}(\xx-(\uu,0))|=t,\xx\in \mathfrak B\}\\
\displaystyle=\left\{\zz\in \mathfrak A:\zz\cdot\frac{(-2\bar A^{-1}\uu,1)}{\nu_{\bar A}(\uu)}=\frac{t^2-|\bar A^{-1}\uu|^2}{\nu_{\bar A}(\uu)}\right\},
\end{array}
$$
where $\nu_{\bar A}(\uu)=\sqrt{1+4|\bar A^{-1}\uu|^2}$.
\end{itemize}

\end{prop}
The map ${\color{black}m_{n,A}}^{-1}$ transforms an ellipsoid into a hyperplane​. 
Changing variables using this map ${\color{black}m_{n,A}}$ plays a critical role in reducing the elliptical Radon transform to the regular Radon transform.
\begin{proof}
We can easily check that ${\color{black}m_{n,A}}^{-1}\circ {\color{black}m_{n,A}}(\zz)=\zz$ for $\zz\in \mathfrak A$ and ${\color{black}m_{n,A}}\circ {\color{black}m_{n,A}}^{-1}(\xx)=\xx$ for $\xx\in \mathfrak B$, so  ${\color{black}m_{n,A}}:A\to \mathfrak B$ is a bijection.

Consider ${\color{black}m_{n,A}}^{-1}(\xx)\cdot(-2\bar A^{-1}\uu,1)/\nu_{\bar A}(\uu)$ for $\xx\in \mathfrak B$ with $|A^{-1}(\xx-(\uu,0))|=t$:
\begin{equation}\label{eq:mn}
\begin{array}{ll}
\displaystyle {\color{black}m_{n,A}}^{-1}(\xx)\cdot\frac{(-2\bar A^{-1}\uu,1)}{\nu_{\bar A}(\uu)}=(\bar A^{-1}\xx',|A^{-1}\xx|^2)\cdot\frac{(-2\bar A^{-1}\uu,1)}{\nu_{\bar A}(\uu)}\\
\displaystyle =\frac{-2\bar A^{-1}\xx'\cdot\bar A^{-1}\uu+|A^{-1}\xx|^2}{\nu_{\bar A}(\uu)}\\
\displaystyle =\frac{|A^{-1}(\xx-(\uu,0))|^2-|\bar A^{-1}\uu|^2}{\nu_{\bar A}(\uu)}=\frac{t^2-|\bar A^{-1}\uu|^2}{\nu_{\bar A}(\uu)}.
\end{array}
\end{equation}
\end{proof}

We define the function $k(\zz)$ on $\RR^n$ by
$$
k(\zz)=\left\{\begin{array}{ll}\displaystyle\frac{f\circ {\color{black}m_{n,A}}(\zz)}{\sqrt{z_n-|\zz'|^2 }}&\mbox{ if }0\leq|\zz'|^2<z_n,\\
0&\mbox{ otherwise,}\end{array}\right.
$$
where $\zz=(\zz',z_n)\in\RR^n$. 
This is equivalent for $x_n>0$, to 
$$
f(\xx)=x_na_n^{-1} k\circ {\color{black}m_{n,A}}^{-1}(\xx)=x_na_n^{-1} k(\bar A^{-1}  \xx',|A^{-1}\xx|^2),
$$
where $\xx=(\xx',x_n)\in\RR^n$.
By the evenness of $f$, we have
\begin{equation}\label{eq:ffromk}
f(\xx)=|x_n|a_n^{-1} k(\bar A^{-1}  \xx',|A^{-1}\xx|^2).
\end{equation}
Let the regular Radon transform $Rk(\boldsymbol{e_\theta},s)$ be defined by
$$
Rk(\boldsymbol{e_\theta},s)=\intL_{\boldsymbol{e_\theta}^\perp}k(s\boldsymbol{e_\theta}+\boldsymbol\eta)d\boldsymbol\eta, \qquad (\boldsymbol{e_\theta},s)\in S^{n-1}\times\RR,
$$
where $s\in\RR$ and for $\boldsymbol\theta=(\theta_1,\theta_2,\cdots,\theta_{n-1})\in[0,2\pi)\times[0,\pi]^{n-2},$
$$
\boldsymbol{e_\theta}=
 \left(\begin{array}c\sin\theta_1\sin\theta_2\cdots\sin\theta_{n-1}\\
\cos\theta_1\sin\theta_2\cdots\sin\theta_{n-1}\\
\cos\theta_2\sin\theta_3\cdots\sin\theta_{n-1}\\
\vdots\\
\cos\theta_{n-2}\sin\theta_{n-1}\\
\cos\theta_{n-1}\end{array}\right)\in S^{n-1}.
$$
This can be represented by
\begin{equation}\label{eq:radondefi}
Rk(\boldsymbol{e_\theta},s)=
\frac1{|\cos\theta_{n-1}|}\intL_{\RR^{n-1}}k\left(\zz',-\frac{\boldsymbol{e_\theta'}\cdot\zz'}{\cos\theta_{n-1}}+\frac s{\cos\theta_{n-1}}\right)d\zz',
\end{equation}
for $\theta_{n-1}\neq\pi/2$ and $\boldsymbol{e_\theta}=(\boldsymbol{e_\theta'},\cos\theta_{n-1})\in S^{n-1}$. 
The following theorem shows the relation between ${\color{black}R_{E,A}}f$ and $Rk$.
\begin{thm}\label{thm:mainrelation}
Let $f\in C(\RR^n)$ be even in $x_n$ and have compact support in $\RR^{n-1}\times\RR\setminus\{0\}$.
Then we have 
\begin{equation}\label{eq:mainrel}
{\color{black}R_{E,A}}f(\uu,t)=\frac{2|\mathbf a|_1t}{\nu_{\bar A}(\uu)}Rk\left(\frac{(-2\bar A^{-1}\uu,1)}{\nu_{\bar A}(\uu)},\frac{t^2-|\bar A^{-1}\uu|^2}{\nu_{\bar A}(\uu)}\right).
\end{equation}
Again, $\nu_{\bar A}(\uu)=\sqrt{1+4|\bar A^{-1}\uu|^2}$.
\end{thm}
\begin{proof}
Combining two equations (\ref{eq:defiref}) and (\ref{eq:ffromk}), we have
$$
{\color{black}R_{E,A}}f(\uu,t)=\displaystyle2t^{n}|\mathbf a|_1\intL_{|\yy'|\leq 1}k(\yy' t+\bar A^{-1}  \uu,|\yy' t+\bar A^{-1} \uu |^2+(1-|\yy'|^2)t^2)d\yy'.
$$
Let us consider the variable change
$$
\zz'=\yy't+\bar A^{-1} \uu\qquad\left(\Longleftrightarrow \yy'=\frac{\zz'-\bar A^{-1}  \uu}t\right).
$$
The Jacobian of this transformation is $t^{1-n}$, so
$$
{\color{black}R_{E,A}}f(\uu,t)=\displaystyle2|\mathbf a|_1t\intL_{|\zz'-\bar A^{-1}\uu|\leq t}k(\zz',|\zz'|^2+t^2-|\zz'-\bar A^{-1}\uu|^2)d\zz'.
$$
Since $k$ has compact support in $\{\zz\in\RR^2:|\zz'|^2<z_n\}$, we have
\begin{equation}\label{eq:refrk}
\begin{array}{ll}
{\color{black}R_{E,A}}f(\uu,t)&\displaystyle=2|\mathbf a|_1t\intL_{\RR^{n-1}}k(\zz',|\zz'|^2+t^2-|\zz'-\bar A^{-1}\uu|^2)d\zz'\\
&=\displaystyle2|\mathbf a|_1t\intL_{\RR^{n-1}}k(\zz',t^2-|\bar A^{-1}\uu|^2+2(\bar A^{-1}\uu)\cdot\zz')d\zz'.
\end{array}
\end{equation}
We recognize the right hand side as the integral along the hyperplane perpendicular to 
$$
(-2\bar A^{-1}\uu,1)/\nu_{\bar A}(\uu)
$$
with (signed) distance
$$
(t^2-|\bar A^{-1}\uu|^2)/\nu_{\bar A}(\uu)
$$
from the origin.
In this case, the measure for the hyperplane becomes $\nu_{\bar A}(\uu)d\zz'.$
Setting 
$$
\boldsymbol{e_\theta}=(-2\bar A^{-1}\uu,1)/\nu_{\bar A}(\uu),\qquad\mbox{and}\qquad s=(t^2-|\bar A^{-1}\uu|^2)/\nu_{\bar A}(\uu)
$$
in equation~\eqref{eq:radondefi} we have the desired formula from equation~\eqref{eq:refrk}.
\end{proof}

Using the projection slice theorem for the regular Radon transform, we obtain an analog of the projection slice theorem for the elliptical Radon transform ${\color{black}R_{E,A}}f$:
\begin{thm}\label{thm:projectionslice}
Let $f\in C^\infty(\RR^n)$ be even in $x_n$ and have compact support in $\RR^{n-1}\times\RR\setminus\{0\}$.
Then we have for $(\boldsymbol\alpha,\beta)\in\RR^{n-1}\times\RR$,
$$
\hat k(\boldsymbol\alpha,\beta)=|\mathbf a|_1^{-1}e^{i\frac{|\boldsymbol\alpha|^2}{4\beta}}\intL^\infty_0{\color{black}R_{E,A}} f\left(-\frac{\bar A \boldsymbol\alpha}{2\beta},t\right)e^{-i\beta t^2}dt
$$
where $\hat k$ is the $n$-dimensional Fourier transform of $k$. 
\end{thm}
\begin{proof}
The projection slice theorem implies
$$
\hat k(\sigma\boldsymbol{e_\theta})=\intR Rk(\boldsymbol{e_\theta},s)e^{-is\sigma}ds.
$$
To use this theorem, we multiply equation~\eqref{eq:mainrel} by $e^{-i\frac{t^2-|\bar A^{-1}\uu|^2}{\nu_{\bar A}(\uu)}\sigma}$ and integrate $\frac{t^2-|\bar A^{-1}\uu|^2}{\nu_{\bar A}(\uu)}$ from $-|\bar A^{-1}\uu|^2/\nu_{\bar A}(\uu)$ to infinity:
$$
\begin{array}{ll}
\displaystyle\intL^\infty_{-|\bar A^{-1}\uu|^2/\nu_{\bar A}(\uu)} Rk\left(\frac{(-2\bar A^{-1}\uu,1)}{\nu_{\bar A}(\uu)},\frac{t^2-|\bar A^{-1}\uu|^2}{\nu_{\bar A}(\uu)}\right)e^{-i\frac{t^2-|\bar A^{-1}\uu|^2}{\nu_{\bar A}(\uu)}\sigma}d\left(\frac{t^2-|\bar A^{-1}\uu|^2}{\nu_{\bar A}(\uu)}\right).
\end{array}
$$
Since $k$ has compact support in $\{\zz\in\RR^n:|\zz'|^2<z_n\}$, the plane perpendicular to $\frac{(-2\bar A^{-1}\uu,1)}{\nu_{\bar A}(\uu)}$ with distance from the origin less than $\frac{-|\bar A^{-1}\uu|^2}{\nu_{\bar A}(\uu)} $ does not intersect the compact support of $k$.
Hence, $\hat k(\sigma/\nu_{\bar A}(\uu)(-2\bar A^{-1}\uu,1))$ is equal to
$$
\begin{array}{ll}
&\displaystyle\intL^\infty_{-|\bar A^{-1}\uu|^2/\nu_{\bar A}(\uu)} Rk\left(\frac{(-2\bar A^{-1}\uu,1)}{\nu_{\bar A}(\uu)},\frac{t^2-|\bar A^{-1}\uu|^2}{\nu_{\bar A}(\uu)}\right)e^{-i\frac{t^2-|\bar A^{-1}\uu|^2}{\nu_{\bar A}(\uu)}\sigma}d\left(\frac{t^2-|\bar A^{-1}\uu|^2}{\nu_{\bar A}(\uu)}\right)\\
&=\displaystyle\intL^\infty_{-|\bar A^{-1}\uu|^2/\nu_{\bar A}(\uu)} \frac{\nu_{\bar A}(\uu)}{2|\mathbf a|_1t} {\color{black}R_{E,A}}f(\uu,t)e^{-i\frac{t^2-|\bar A^{-1}\uu|^2}{\nu_{\bar A}(\uu)}\sigma}d\left(\frac{t^2-|\bar A^{-1}\uu|^2}{\nu_{\bar A}(\uu)}\right)\\
&=\displaystyle\frac{1}{|\mathbf a|_1}\intL^\infty_{0}  {\color{black}R_{E,A}}f(\uu,t)e^{-i\frac{t^2-|\bar A^{-1}\uu|^2}{\nu_{\bar A}(\uu)}\sigma}dt,
\end{array}
$$
where in the second line, we used Theorem~\ref{thm:mainrelation}.
Setting $\boldsymbol\alpha=-2\sigma\bar A^{-1}\uu/\nu_{\bar A}(\uu)\in\RR^{n-1}$ and $\beta=\sigma/\nu_{\bar A}(\uu)\in\RR$ completes our proof.
\end{proof}

Taking the inverse Fourier transform of a function $k$ and using equation (\ref{eq:ffromk}), we have the following inversion formula:
\begin{cor}
Let $f\in C^\infty(\RR^n)$ be even in $x_n$ and have compact support in $\RR^{n-1}\times\RR\setminus\{0\}$.
Then $f(\xx)$ can be reconstructed as follows:
\begin{equation}\label{eq:kinversion}
 \displaystyle\frac{(-1)^{n-1}|x_n| }{2\pi^{n} |\mathbf a|_1^2}\intL_{\RR^{n-1}}\left.\Lambda_t ({\color{black}R_{E,A}}f)( \boldsymbol\alpha, t)\right|_{t=|\bar A^{-1}(\boldsymbol\alpha-\xx')|^2+a_n^{-2}x_n^2}d\boldsymbol\alpha,
\end{equation}
where 
$$
\Lambda_t h(t)=\intR\intL\half e^{i(t^2-\tau^2)\beta}|\beta|^{n-1}h(\tau)d\tau d\beta.
$$
\end{cor}

\begin{proof}
We get for $x_n>0$,
$$
  \begin{array}{ll}
   f(\xx)&=\displaystyle\frac{x_n }{a_n} k(\bar A^{-1}\xx',|A^{-1}\xx|^2)\\
&=\displaystyle\displaystyle\frac{x_n }{(2\pi)^n a_n}\intL_{\RR^{n-1}} \intR e^{i\bar A^{-1}\xx' \cdot \boldsymbol\alpha}e^{i\beta |A^{-1}\xx|^2}\hat k(\boldsymbol\alpha,\beta) d\beta d\boldsymbol\alpha\\
&=\displaystyle\displaystyle\frac{x_n }{(2\pi)^n a_n|\mathbf a|_1}\intL_{\RR^{n-1}}\intR e^{i\frac{|\boldsymbol\alpha|^2}{4\beta}}e^{i\bar A^{-1}\xx' \cdot \boldsymbol\alpha}e^{i\beta |A^{-1}\xx|^2}G\left(-\frac{\bar A \boldsymbol\alpha}{2\beta}, \beta\right) d\beta d\boldsymbol\alpha,
  \end{array}
$$
 where in the first line, we used equation (\ref{eq:ffromk}) and in the last line, we used Theorem~\ref{thm:projectionslice} and 
 $$
G(\uu,w)=\intL^\infty_0{\color{black}R_{E,A}} f(\uu,t)e^{-iwt^2}dt.
$$
 The evenness of $f$ in $x_n$ gives us
 $$
f(\xx)= \displaystyle\frac{|x_n| }{(2\pi)^n a_n|\mathbf a|_1}\intL_{\RR^{n-1}}\intR e^{i\frac{|\boldsymbol\alpha|^2}{4\beta}} e^{i\boldsymbol\alpha\cdot(\bar A^{-1}\xx')}e^{i\beta |A^{-1}\xx|^2}G\left(-\frac{\bar A \boldsymbol\alpha}{2\beta}, \beta\right) d\beta d\boldsymbol\alpha .
$$
Changing the variables $-\bar A\aalpha/(2\beta)\to\aalpha$, $f(\xx)$ is equal to
$$
 \displaystyle\frac{(-2)^{n-1}|x_n| }{(2\pi)^n |\mathbf a|_1^2}\intL_{\RR^{n-1}}\intR e^{i\beta(|\bar A^{-1}\boldsymbol\alpha|^2-2(\bar A^{-1}\boldsymbol\alpha)\cdot(\bar A^{-1}\xx')+|A^{-1}\xx|^2)}G\left( \boldsymbol\alpha, \beta\right) |\beta|^{n-1} d\beta d\boldsymbol\alpha .
 $$
By definition of $G$, $f(\xx)$ can be determined through
$$
 \displaystyle\frac{(-1)^{n-1}|x_n| }{2\pi^{n} |\mathbf a|_1^2}\intL_{\RR^{n-1}}\intR\intL\half e^{i\beta(|\bar A^{-1}\boldsymbol\alpha-\bar A^{-1}\xx'|^2+a_n^{-2}x_n^2-t^2)}{\color{black}R_{E,A}}f( \boldsymbol\alpha, t) |\beta|^{n-1}dtd\beta d\boldsymbol\alpha  .
$$

\end{proof}
When $n=2$, equation \eqref{eq:kinversion} becomes
$$
 \begin{array}{ll}
 f(\xx)&=\displaystyle\frac{|x_2| }{2\pi^{2} a_1^2a_2^2}\intR\intR\intL\half e^{i\beta(|a_1^{-1}(\alpha-x_1)|^2+a_2^{-2}x_2^2-t^2)}{\color{black}R_{E,A}}f(\alpha, t) |\beta|dtd\beta d\alpha\\
   &=-\displaystyle\frac{|x_2| }{\pi^{2} a_1^2a_2^2}\intR\intL\half \frac{{\color{black}R_{E,A}}f(\alpha, t)}{(|a_1^{-1}(\alpha-x_1)|^2+a_2^{-2}x_2^2-t^2)^2}dt d\alpha,
   \end{array}
$$
since in the distribution sense,
$$
\intR e^{ix\beta}|\beta|d\beta=-\frac{2}{x^2}.
$$
\section{2-dimensional numerical implementation}
Here we discuss the results of 2-dimensional numerical implementations.

In 2-dimension, equation~\eqref{eq:mainrel} becomes
\begin{equation}\label{eq:refrk2}
{\color{black}R_{E,A}}f(u,t)=\frac{2 a_1a_2t}{\sqrt{1+4|a_1^{-1}u|^2}}Rk\left(\frac{(-2a_1^{-1}u,1)}{\sqrt{1+4|a_1^{-1}u|^2}},\frac{t^2-|a_1^{-1}u|^2}{\sqrt{1+4|a_1^{-1}u|^2}}\right).
\end{equation}
Setting $(\cos\theta,\sin\theta)=(-2a_1^{-1}u,1)/\sqrt{1+4|a_1^{-1}u|^2}\in S^{1}$ and 
$$
s=\frac{t^2-|a_1^{-1}u|^2}{\sqrt{1+4|a_1^{-1}u|^2}} 
$$
in equation~(\ref{eq:refrk2}), we have
\begin{equation*}
 2R k(\cos\theta,\sin\theta,s)=\displaystyle\frac{|\csc\theta|}{a_1a_2h(s,\theta)}{\color{black}R_{E,A}}f\left(-\frac{a_1\cot\theta}2,h(s,\theta)\right)\mbox{ if }s\csc\theta>-\frac{\cot^2\theta}4
\end{equation*}
where  
$$
h(s,\theta)=\sqrt{s\csc\theta+\frac{\cot^2\theta}4}.
$$ 
(While $Rk(\sin\theta,\cos\theta,s)$ is used in Theorem~\ref{thm:mainrelation}, $Rk(\cos\theta,\sin\theta,s)$ is used to utilize the built-in function in \textsc{Matlab}. 
Hence a small change is required.)
Again, since $k$ has compact support in $\{\zz\in\RR^2:|z_1|^2<z_2\}$, the plane perpendicular to $\boldsymbol{e_\theta}=(\cos\theta,\sin\theta)$ with distance from the origin less than $-\frac{\cot^2\theta}{4\csc\theta}$ does not intersect the compact support of $k$.
Therefore, we have
\begin{equation}\label{eq:mainrelation}
 2R k(\boldsymbol{e_\theta},s)=\left\{\begin{array}{ll}\displaystyle\frac{|\csc\theta|}{a_1a_2h(s,\theta)}{\color{black}R_{E,A}}f\left(-\frac{a_1\cot\theta}2,h(s,\theta)\right)&\mbox{ if }s\csc\theta>-\dfrac{\cot^2\theta}4,\\
0&\mbox{otherwise,}\end{array}\right.
\end{equation}
First of all, $a_1$ and $a_2$ were set to be 0.8 and 1, respectively. 
In the experiments presented here we use the phantom shown in Figure~\ref{fig:spherical} (a).
The phantom, supported within the rectangle $[-1,1]\times[-1,1]$, is the sum of eight characteristic functions of disks.
Notice that it has to be even with respect to the $x_2$-axis and there are four characteristic functions of disks centered at $(0.2,0.4)$, $(0,0.5)$, $(-0.3,0.3)$, and $(-0.5,0.2)$ with radii 0.2 0.15 0.05 and 0.05, whose values are 1, 0.5, 1.5, and 2 above the $x_1$-axis.
Hence we also include their reflection below the $x_1$-axis.
(Actually, our phantom has support in 
$$
\{\xx\in\RR^2:(x_1/a_1)^2+((x_1/a_1)^2+x_2^2)^2<1\}.
$$
This implies that $k$ has support in the unit ball and this makes it sufficient to consider the range $[-1,1]$ in $s$.) 
The $256\times256$ images are used in Figure~\ref{fig:spherical}.
To reconstruct the image in Figure~\ref{fig:spherical} (b), we have $256\times256$ projections for $\theta\in[0,2\pi]$ and $s\in[-1,1]$ in equation~(\ref{eq:mainrelation}).
All projections are computed by numerical integration.
After finding the function $k$ using the usual inversion code for the regular Radon transform, we obtain the function $f$ using equation~(\ref{eq:ffromk}).
(When using the inversion code for the regular Radon transform, the built-in function ``iradon'' in \textsc{Matlab} was used. 
The function ``iradon'' is the inversion of the built-in function ``radon'' in \textsc{Matlab} which considers the number of the pixels where the line passes through.
Thus when computing ${\color{black}R_{E,A}}f$, we also considers the number of the pixels where the ellipse passes through.
We used the default version of the function in which the filter, whose aim is to deemphasize high frequencies, is set to the Ram-Lak filter and the interpolation is set to be linear.)

While Figure~\ref{fig:spherical} (b) demonstrates the image reconstructed from the exact data, Figure~\ref{fig:spherical} (c) shows the absolute value of the reconstruction from noisy data.
The noise is modeled by normally distributed random numbers and this is scaled so that its norm was equal to 5$\%$ of the norm of the exact data.
In Figure~\ref{fig:spherical} (c) the noisy data is modeled by adding the noise values scaled to 5$\%$ of the norm of the exact data to the exact data.
In Figure~\ref{fig:sphericalsurface}, the surface plots of Figure~\ref{fig:spherical} (a) and (b) are provided.
Another phantom and reconstruction are shown in Figure~\ref{fig:sphericalphantom}.
The reconstructed two dimensional data sets consisting of 256$\times$256 projections using the implemented inversion formula have been computed in less than one second (around 0.3 second) on a Intel(R) Cor(TM) i5-3470 CPU @ 3.20 GHz.
\begin{figure}[here]
\begin{center}
        \subfigure[]{
            \includegraphics[width=0.3\textwidth]{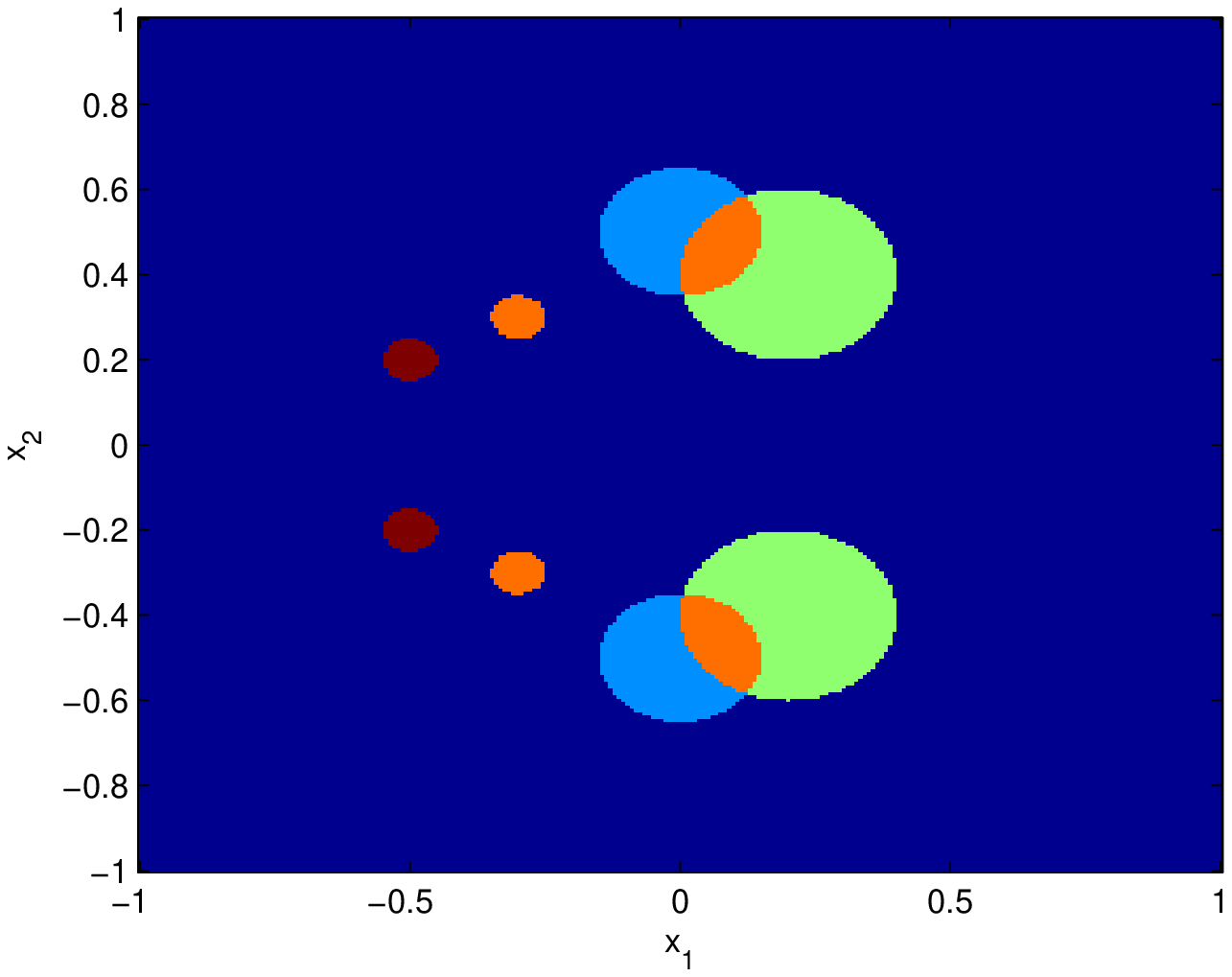}        }%
        \subfigure[]{
     \includegraphics[width=0.3\textwidth]{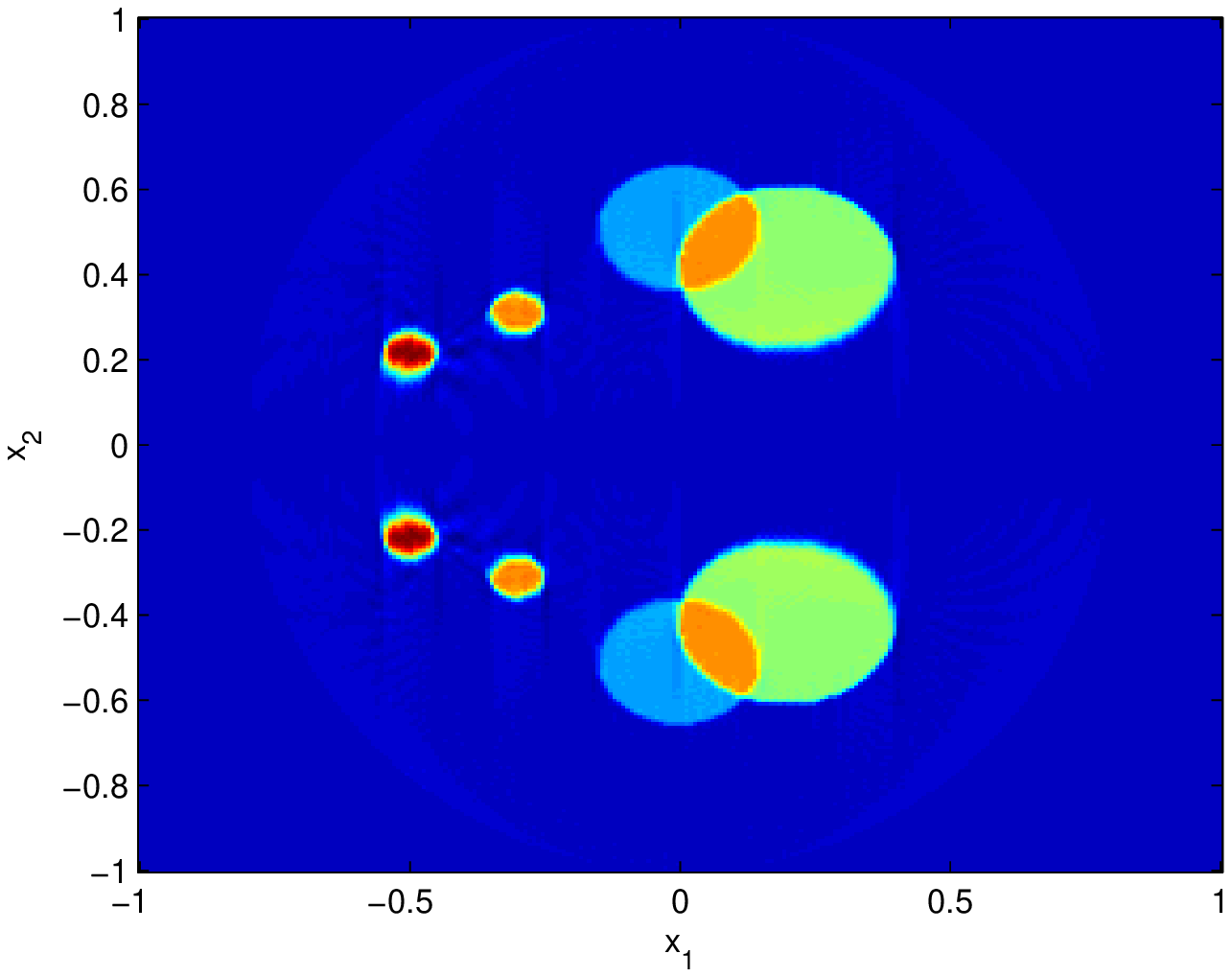}}
             \subfigure[]{
     \includegraphics[width=0.3\textwidth]{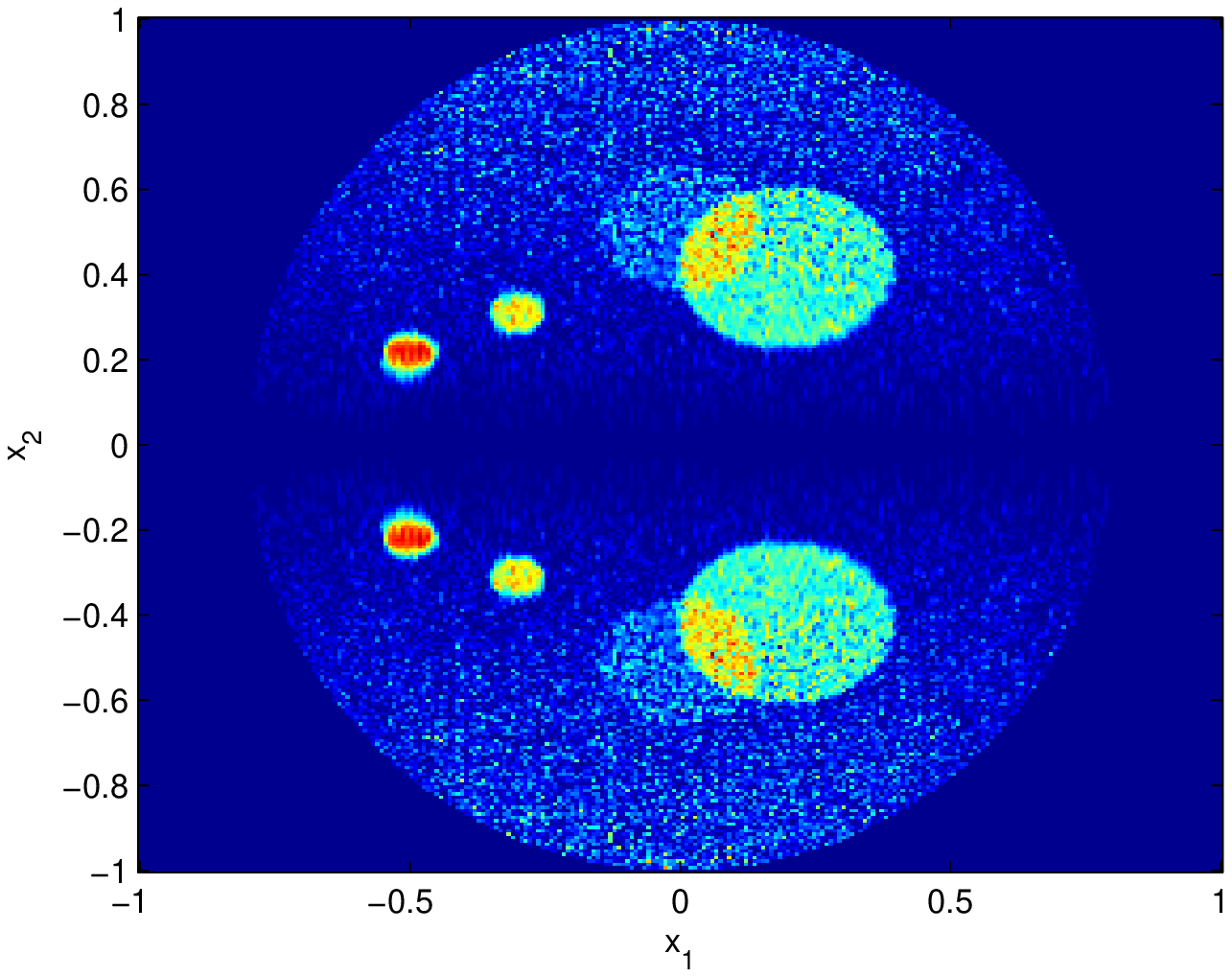}}
     \end{center}
\caption{Reconstructions in 2 dimensions: (a) the phantom, (b) the reconstruction form exact data, and (c) the reconstruction from noisy data}
\label{fig:spherical}
\end{figure} 
\begin{figure}[here]
\begin{center}
        \subfigure[]{
            \includegraphics[width=0.45\textwidth]{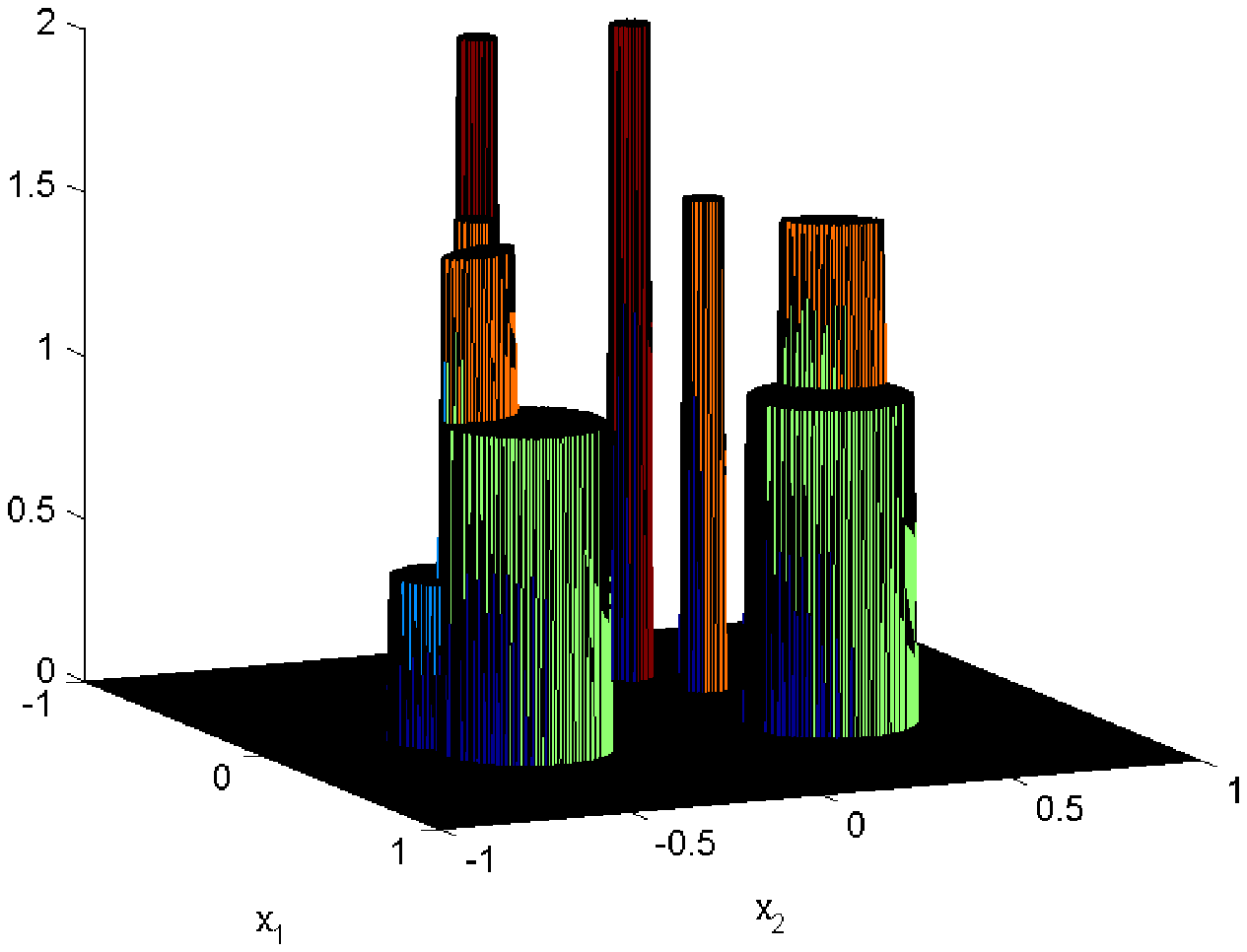}        }%
        \subfigure[]{
     \includegraphics[width=0.45\textwidth]{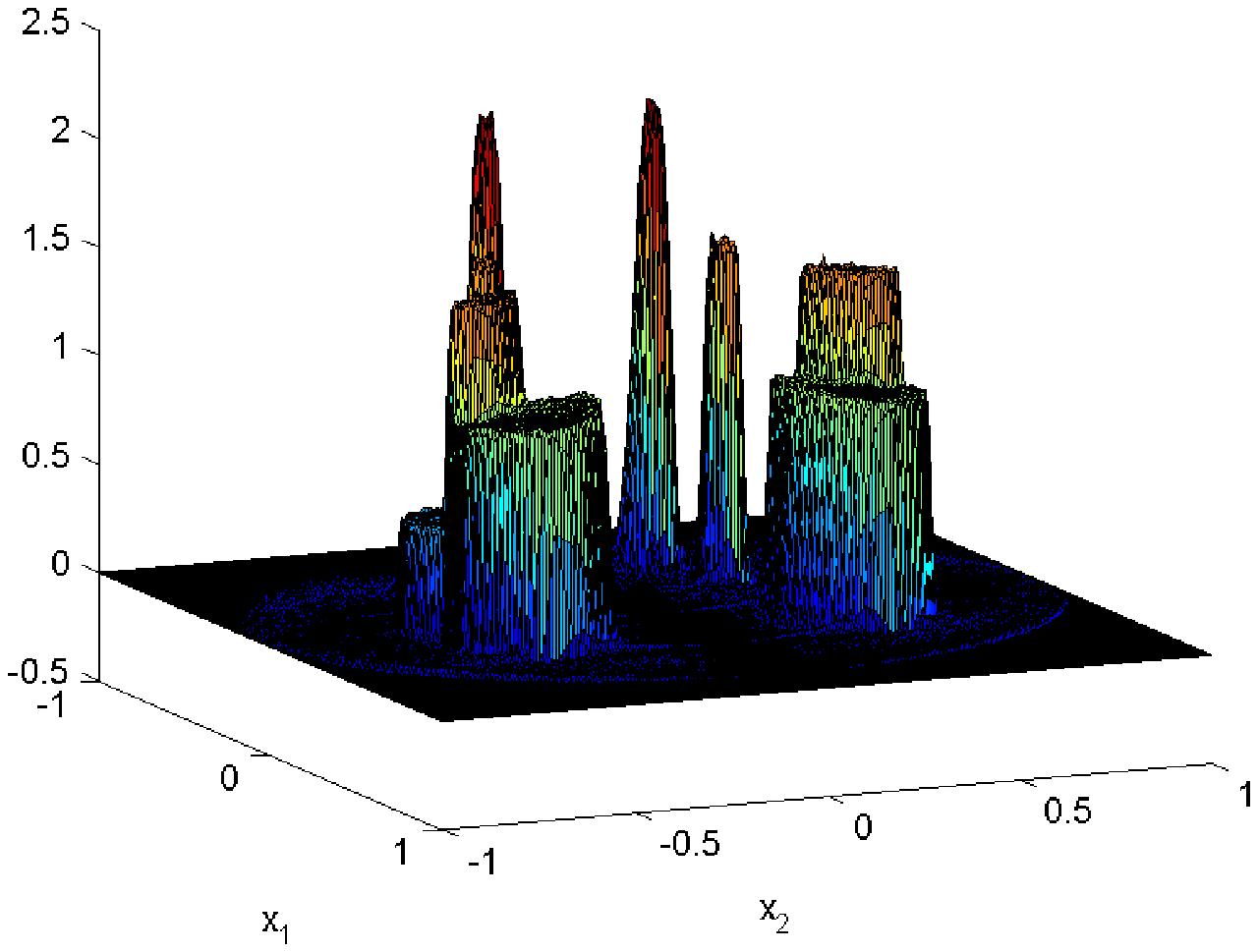}}

\end{center}
\caption{Surface plots (a) the phantom and (b) the reconstruction from exact data}
\label{fig:sphericalsurface}
\end{figure} 
\begin{figure}[here]
\begin{center}
        \subfigure[]{
            \includegraphics[width=0.45\textwidth]{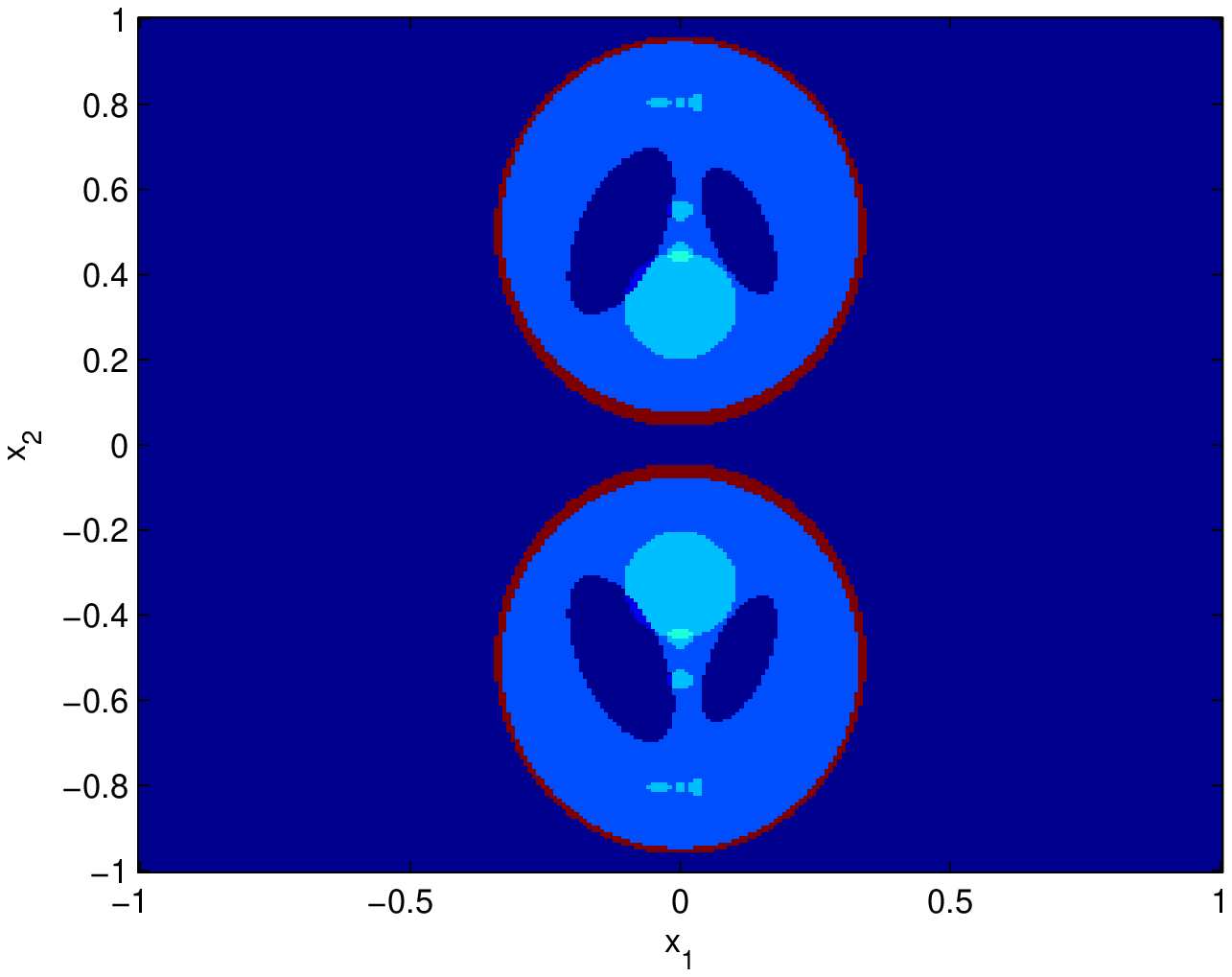}        }%
        \subfigure[]{
     \includegraphics[width=0.45\textwidth]{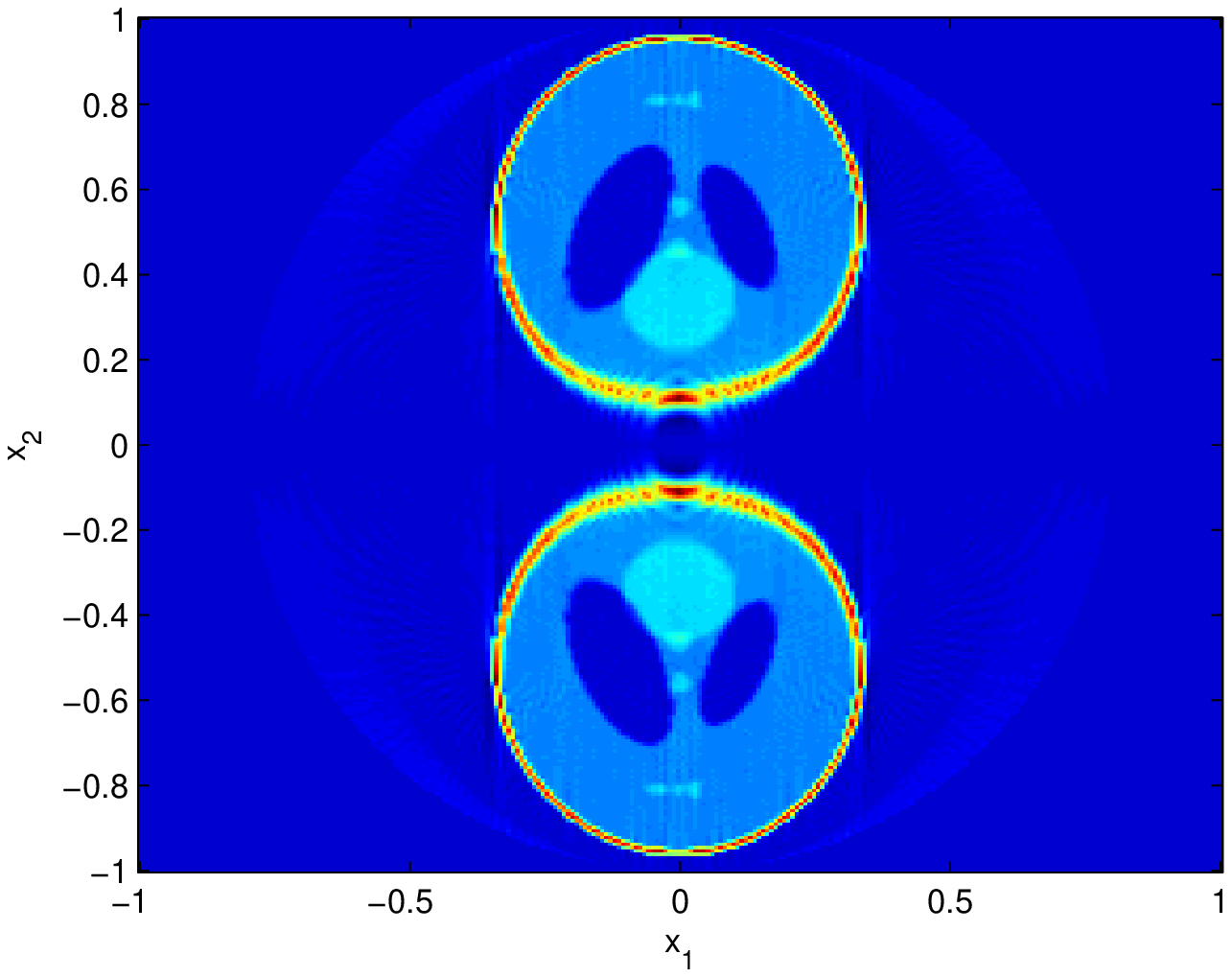}}
     \end{center}
\caption{Reconstruction in 2D: (a) the phantom and (b) the reconstruction from exact data }
\label{fig:sphericalphantom}
\end{figure} 
\section{Conclusion}
This paper is devoted to the study of the elliptical Radon transform arising in seismic imaging.
We provide an inversion formula for the elliptical Radon transform by reducing this transform to the regular Radon transform.
Also, we demonstrate our algorithm by providing numerical simulations.

\section*{Acknowledgements}
The first author thanks C. Jung for fruitful discussions.
The first author was supported in part by Basic Science Research Program through the National Research Foundation of Korea(NRF) funded by the Ministry of science, ICT and future planning (2015R1C1A1A02037662)).

\end{document}